\documentclass[11pt,oneside]{amsart}
\usepackage{amscd}
\usepackage{amssymb}
\usepackage{amsfonts}
\usepackage[centertags]{amsmath}
\usepackage{latexsym}
\usepackage{verbatim}
\usepackage{amsthm}
\usepackage{color}
\numberwithin{equation}{section}
\newcommand{\n}{\frak n}

\theoremstyle{plain}
\newtheorem{theorem}{Theorem}[section]

\newtheorem{lemma}[theorem]{Lemma}
\newtheorem{prop}[theorem]{Proposition}
\newtheorem{cor}[theorem]{Corollary}
\newtheorem{rem}[theorem]{Remark}
\newtheorem{defn}[theorem]{Definition}
\theoremstyle{definition}
\newtheorem{ex}{Example}

\newcommand\C{{\mathbb C}}
\newcommand\R{{\mathbb R}}

\newcommand\Q{{\mathbb Q}}

\renewcommand{\bar}[1]{\overline{#1}}

\newcommand{\g}{\mathfrak{g}}

\begin{document}

\title[SKT and tamed symplectic structures on solvmanifolds]{SKT and tamed symplectic structures on solvmanifolds}
\author{Anna Fino,  Hisashi Kasuya and Luigi Vezzoni}
\address{(Anna Fino, Luigi Vezzoni) Dipartimento di Matematica G. Peano \\ Universit\`a di Torino\\
Via Carlo Alberto 10\\
10123 Torino\\ Italy}
\email{annamaria.fino@unito.it}
 \email{luigi.vezzoni@unito.it}
 \address{(Hisashi Kasuya) Department of Mathematics, Tokyo Institute of 
Technology, 1-12-1, O-okayama, Meguro, Tokyo 152-8551, Japan}
\email{khsc@ms.u-tokyo.ac.jp}

\thanks{This work was partially supported by the project PRIN {\em Variet\`a reali e complesse: geometria, topologia e analisi armonica},  the project FIRB {\em Differential Geometry and Geometric functions theory} and GNSAGA (Indam) of Italy.
}

\begin{abstract}
We study  the existence of  strong K\"ahler with torsion (SKT) metrics and of  symplectic  forms  taming   invariant  complex structures $J$  on solvmanifolds $G/\Gamma$  providing some negative results for some classes of solvmanifolds. In particular,  we show that   if   either  $J$ is invariant  under the action of  a nilpotent complement   of the nilradical of $G$ or $J$ is abelian or $G$ is almost abelian (not of  type (I)),   then the  solvmanifold $G/\Gamma$  cannot admit any  symplectic  form   taming the complex structure $J$, unless $G/\Gamma$ is K\"ahler. As a consequence,  we show  that     the  family  of non-K\"ahler complex manifolds constructed by Oeljeklaus and Toma cannot admit any symplectic form taming the complex structure.

\end{abstract}

\maketitle

\section{Introduction}
A symplectic form $\Omega$ on a complex manifold $(M,J)$ is said  {\em taming}  the complex structure $J$  if
$$
\Omega(X,JX)>0
$$
for any non-zero vector field $X$ on $M$ or, equivalently, if the $(1,1)$-part of $\Omega$ is positive.  The pair $(\Omega, J)$ was called  in \cite{ST} a  {\em Hermitian-symplectic} structure and it was shown that these structures   appear  as static solutions  of  the so-called {\em pluriclosed flow}.    By \cite{LZ, ST}   a  compact complex surface admitting a Hermitian-symplectic structure is necessarily K\"ahler (see also Proposition 3.3 in \cite{DLZ}) and it follows from \cite{P} that non-K\"ahler Moishezon complex structures on compact manifolds cannot be tamed by a symplectic form (see also \cite{Z}).
However, it  is  still  an open problem to find out an example of a compact Hermitian-symplectic manifold non admitting
K\"ahler structures.  It is well known that  Hermitian-symplectic structures can be viewed  as  special strong K\"ahler with torsion structures (\cite{EFV}) and that their existence can be characterized in terms of currents (\cite{V}). Here  we recall that a Hermitian metric is called {\em strong K\"ahler with torsion} (SKT) if its fundamental form is $\partial \bar\partial$-closed (see  for instance \cite{FT,Cavalcanti} and the references therein).
SKT nilmanifolds were first  studied in \cite{FPS}  in six dimension and recently in \cite{EFV} in any dimension, where by {\em nilmanifold} we mean a compact quotient of a simply connected nilpotent Lie group $G$ by a co-compact lattice $\Gamma$.  Very few results are known for   the existence of SKT  metrics on solvmanifolds endowed with an invariant complex structure. By {\em solvmanifold}  $G/\Gamma$ we mean a  compact quotient of  a simply connected solvable Lie group $G$  by  a  lattice $\Gamma$  and by  {\em invariant complex structure} on  $G/\Gamma$  we mean a complex structure induced by a left invariant complex structure on $G$. We will  call  a solvmanifold endowed with an invariant complex structure a  {\em complex solvmanifold}.

\medskip

From \cite{EFV}   it is known that  a nilmanifold $G/\Gamma$  endowed with  an  invariant complex structure $J$  cannot admit any symplectic  form taming $J$  unless it  admits a K\"ahler structure (or equivalently  $G/\Gamma$ is a complex torus). Then  it is quite  natural trying to extend the result  to complex  solvmanifolds.

 By \cite{Hasegawa} a  solvmanifold  $G/\Gamma$ admits a K\"ahler structure if and only if  it is a  finite quotient of a complex torus. This in particular implies that when $G$ is not of  type (I) and non abelian, then $G/\Gamma$ is not K\"ahler. We recall that being of  type (I)  means that  for any $X\in \g$  all eigenvalues of the adjoint operator $ad_{X}$ are pure imaginary.

Given a solvable Lie algebra $\g$ we denote by $\n$ its {\em nilradical} which is defined as the {\em maximal nilpotent ideal} of $\g$.
It is well known that  there always exists  a {\em nilpotent complement} $\mathfrak c$ of $\frak n$  in $\frak g$, i.e. there exists a nilpotent subalgebra $\frak c$ of $\g$ such that $\g=\mathfrak c+\n$ (see \cite[Theorem 2.2]{Dek}).
In general the complement $\mathfrak c$ is not unique and we do not expect to have a direct sum between $\mathfrak c$ and $\n$.

\medskip
The first main result of the paper consists in proving the following theorem about the non-existence of Hermitian-symplectic and SKT structures  on homogeneous spaces of splitting Lie groups.

\begin{theorem}\label{spskh}
Let  $G$  be  a   Lie group endowed with a left-invariant complex structure $J$  and suppose that

\begin{enumerate}

\item[1)]  the Lie algebra  $\g$ of $G$  is a semidirect  product  $\g=\frak s\ltimes_{\phi}\frak h$, where $\frak s$ is  a solvable Lie algebra and $\frak h$ a Lie algebra;

 \item[2)] $\phi:\frak s\to {\mbox {Der}} \,  (\frak h)$ is a representation on the space of derivations of $\frak h$;

\item[3)] $\phi$ is  not of   type  (I) and the  image $\phi(\frak s)$ is a nilpotent subalgebra of ${\mbox {Der}} \,  (\frak h)$;

\vspace{0.1cm}
\item[4)] $J(\frak h)\subset  {\frak h}$;

\vspace{0.1cm}
\item[5)] $J_{\vert\frak h}\circ \phi(X)=\phi(X)\circ J_{\vert\frak h}$ for any $X\in \frak s$.
\end{enumerate}
Then $\g$ does not admit any  symplectic structure taming $J$.
\noindent Moreover if $\frak s$ is nilpotent and $J(\frak s)\subset \frak s$,
then $\g$ does not admit any $J$-Hermitian  SKT metric.

\end{theorem}
The previous theorem can be in particular applied to compact homogeneous complex spaces of the form $(G/\Gamma,J)$, where $(G,J)$ satisfies conditions $1), \ldots, 5)$ in the theorem and $\Gamma$ is a discrete  subgroup  of $G$. This type of homogeneous spaces covers a large class of examples including  the so-called Oeljeklaus-Toma manifolds (see \cite{OT}).

\medskip
In general  a simply connected solvable Lie group is not of splitting type (i.e. its Lie algebra does not satisfy conditions $1), 2), 3)$ of Theorem \ref{spskh}). The following theorem provides a non-existence result in the non-splitting case.

\begin{theorem}\label{th4.2}
Let $(G/\Gamma,J)$ be a complex solvmanifold.
Assume that $J$ is invariant under the action of  a nilpotent complement of the nilradical $\frak n$.
Then  $G/\Gamma$ admits a   symplectic form taming $J$  if and only if  $(G/\Gamma,J)$  is   K\"ahler.
\end{theorem}

A special class of invariant complex structures on solvmanifolds is provided by {\em abelian complex structures} (see \cite{BDM}).
A complex structure $J$ on a  Lie algebra $\frak g$   is called {\em abelian} if $[JX,JY]=[X,Y]$ for every  $X,Y\in \g$. In the abelian case  the Lie  subalgebra $\g^{1,0}$ of   the complexification $\frak g_{\C}$ of  $\frak g$  is   abelian   and that motivates the name.   In Section \ref{abilian} we will prove the following

\begin{theorem}\label{thABILIAN}
Let $(G/\Gamma, J)$ be a solvmanifold  endowed with an invariant  {\em abelian} complex structure $J$. Then  $(G/\Gamma, J)$  doesn't admit a symplectic form taming $J$  unless it is a complex torus.
\end{theorem}
In the last  section of the paper we take into account solvmanifolds  $G/\Gamma$  with $G$ {\em almost-abelian}. The almost-abelian condition means that the nilradical  $\frak n$ of the Lie algebra  $\g$ of $G$ has codimension $1$ and  $\frak n$  is abelian.   About this case we will prove the following

\begin{theorem}\label{almost-Abilian}
Let $(G/\Gamma,J)$ be a complex solvmanifold with $G$   almost-abelian. Assume $\g$  being  either not  of  type (I)  or $6$-dimensional. Then
 $(G/\Gamma,J)$ does not admit any symplectic form taming $J$.
\end{theorem}

\noindent
 {\em{Acknowledgements}}. The authors are grateful to Weiyi Zhang for useful observations on the previous version of the present paper. 

\section{Preliminary results on representations of Lie algebras}
In this section we prove some preliminary results which will be useful in the sequel.
\subsection{Representations of solvable Lie algebras}
Let $\frak g$ be a solvable Lie algebra and
let $\rho:\frak g\to \mathop{\rm End} (V)$ be a representation on a real vector space $V$ whose image $\rho(\g)$ is a nilpotent subalgebra of $\mathop{\rm End} (V)$.
For every $X\in \frak g$  we can consider the Jordan decomposition
$$
\rho(X)=(\rho(X))_{s}+(\rho(X))_{n}
$$
which induces two maps $\rho_s$ and $\rho_n$ from $\g$ onto $\mathop{\rm End}(V)$.
The following facts can be easily deduced from \cite{dek}:
\begin{itemize}
\item  The maps $\rho_{s}:\g\ni X\mapsto (\rho(X))_{s}\in \mathop{\rm End} (V)$ and $\rho_{n}:\g\ni X\mapsto (\rho(X))_{n}\in \mathop{\rm End} (V)$ are Lie algebra homomorphisms.

\vspace{0.2cm}
\item The images $\rho_{s}(\g)$ and $\rho_{n}(\g)$ are subalgebras of $\mathop{\rm End} (V)$ satisfying $[\rho_{s}(\g),\rho_{n}(\g)]=0$.

\end{itemize}
For a real-valued character $\alpha$ of $\g$, we denote
$$
V_{\alpha}(V)=\{v\in V\,\,:\,\,  \rho_{s}(X)v=\alpha(X)v\mbox{ for every }X\in \frak g \},
$$
and for a complex-valued character $\alpha$ of $\g$ we set
$$
V_{\alpha}(V_{\C})=\{v\in V_{\C}\,\,:\,\, \rho_{s}(X)v=\alpha(X)v \mbox{ for every }X\in \frak g\}.
$$
When $\alpha$ is real we have $V_{\alpha}(V_\C)=V_{\alpha}(V)\otimes \C$.
From the condition  $[\rho_{s}(\g),\rho_{n}(\g)]=0$, we get
$$
\rho(X)\big(V_{\alpha}(V_\C)\big)\subset V_{\alpha}(V_\C),
$$
for any $X\in \frak g$ (see \cite{OV}).
Moreover, as a consequence of  the  Lie theorem, there exits a basis of $V_{\alpha}(V_\C)$ such that for any $X\in \frak c$ the map $\rho(X)$ is represented by an upper triangular matrix
\[\left(
\begin{array}{ccc}
\alpha& & \ast\\
&    \ddots&  \\
0&&\alpha
\end{array}
\right).
\]
Therefore we obtain  a decomposition
\[V_\C= V_{\alpha_{1}}(V_\C)\oplus \dots\oplus V_{\alpha_{n}}(V_\C)
\]
 with  $\alpha_{1},\dots,\alpha_{n}$  characters of $\frak g$.
Since $\rho$ is a real-valued representation, the set $\{\alpha_{1},\dots,\alpha_{n}\}$ is invariant under complex conjugation (i.e., $\bar\alpha_{i}\in \{\alpha_{1},\dots,\alpha_{n}\}$).
 We recall the following
\begin{defn}\label{Itype}{\rm 
A representation  $\rho$ of $\g$  is of {\em type (I)} if for any $X\in \frak g$  all the eigenvalues of  $\rho(X)$ are pure imaginary.}
\end{defn}

The following lemma will be very useful in the sequel:
\begin{lemma}\label{kll22}
Let $\frak h$ and $\g$ be Lie algebras with $\frak g$ solvable.
Let $\rho:\frak g\to D (\frak h)$ be a representation on the space of derivations on  $\frak h$ which we assume to not be of type (I). Then
there exists a complex character $\alpha $ of $\frak g$ satisfying
\begin{equation}\label{hisashicondition}
\mathop{\rm Re}  (\alpha)\not=0\,,\quad V_{\alpha}(\frak h_\C)\not=0 \mbox{ and }[V_{\alpha}(\frak h_\C), V_{\bar\alpha}(\frak h_\C)]=0\,.
\end{equation}
\end{lemma}
\begin{proof}
Since  $\rho$ is assumed to be not of  (I) type, then there exits a complex character
$\alpha_{1}$ such that $\mathop{\rm Re} ( \alpha_{1})\not=0$ and $V_{\alpha_{1}}(\frak h_\C)\not=0$.
If $[V_{\alpha_{1}}(\frak h_\C),V_{\bar\alpha_{1}}(\frak h_\C)]=0$, then $\alpha_1$ satisfies the three conditions required.
Otherwise, since  $\rho_{s}:{\frak g}\ni X\mapsto (ad_{X}) _{s}\in D(\frak h)$, we have $0\not=[V_{\alpha_{1}}(\frak h_\C),V_{\bar\alpha_{1}}(\frak h_\C)]\subset V_{\alpha_{1}+\bar\alpha_{1}}(\frak h_\C)\not=0$ and we take
 $\alpha_{2}=\alpha_{1}+\bar\alpha_{1}=2\mathop{\rm Re} (\alpha_{1})$. Again
if $[V_{\alpha_{2}}(\frak h_\C),V_{\bar\alpha_{2}}(\frak h_\C)]=0$, then $\alpha_{2}$  satisfies all the conditions required, otherwise   we have $0\not=[V_{\alpha_{2}}(\frak h_\C),V_{\alpha_{2}}(\frak h_\C)]\subset V_{2\alpha_{2}}(\frak h_\C)\not=0$ and we
consider $\alpha_{3}=2\alpha_{2}$. We claim that we can iterate this operation until  we get a character $\alpha_k$ satisfying
\eqref{hisashicondition}. Indeed, since $\frak h$ is finite dimensional, we have a sequence of characters
\[\alpha_{2},\; \alpha_{3}=2\alpha_{2},\;\alpha_{4}=2\alpha_{3},\;\dots, \;\alpha_{k}=2\alpha_{k-1}\]
such that $V_{\alpha_{s}}(\frak h_\C)\not=0$ and  $[V_{\alpha_{s}}(\frak h_\C),V_{\alpha_{s}}(\frak h_\C)]\not= 0$ for $2\le s\le k-1$, and $V_{\alpha_{k}}(\frak h_\C)\not=0$ and  $[V_{\alpha_{k}}(\frak h_\C),V_{\alpha_{k}}(\frak h_\C)]= 0$.
Hence the claim follows.
\end{proof}

\subsection{Nilpotent complements of nilradicals of solvable Lie algebras}\label{sol}

Let $\g$ be a solvable Lie algebra with nilradical $\frak n$. As remarked in the introduction
there always exists a nilpotent subalgebra $\frak c$ of $\g$ such that $\g={\frak c}+{\frak n}$  (not necessarily a direct sum) (see \cite[Theorem 2.2]{Dek}).
Such a nilpotent subalgebra  $\frak c$  is called a {\em nilpotent complement} of $\frak n$.
Let us consider  $ad: {\frak c} \to  {\rm {Der}}(\g)$ and the  semisimple
$ad_{s}:{\frak c}\ni C\mapsto (ad_{C}) _{s}\in  {\rm {Der}} (\g)$ and the nilpotent part $ad_{n}:{\frak c}\ni C\mapsto (ad_{C}) _{n}\in {\rm {Der}} (\g)$ of $ad$.
Then $ad_{s}$ and $ad_{n}$ are homomorphisms from $\frak c$.
Since  $\ker ad_{s}=\frak c/\frak c\cap \frak n\cong \frak g/\frak n$, $ad_{s}$ can be  regarded as a homomorphism from $\g$.
For a real-valued character $\alpha$ of $\frak g$, we denote
$$
V_{\alpha}(\g)=\{X\in \g\,\,:\,\, ad_{sY}X=\alpha(Y)X \mbox{ for every }Y\in \g\},
$$
and for a complex-valued character $\alpha$,
\[V_{\alpha}(\g_\C)=\{X\in \g_\C\,\,:\,\,  ad_{sY}X=\alpha(Y)X\mbox{ for every }Y\in \g\}.
\]
If $\alpha$ is real valued we have $V_{\alpha}(\g_\C)=V_{\alpha}(\g)\otimes \C$.
Since $\frak c$ is nilpotent, we have   $ad_{C}(V_{\alpha}(\g_\C))\subset V_{\alpha}(\g_\C)$ for any $C\in \frak c$.
We can take a basis of $V_{\alpha}(\g_\C)$ such that $ad_{C}$ is represented as an upper triangular matrix
\[\left(
\begin{array}{ccc}
\alpha& & \ast\\
&    \ddots&  \\
0&&\alpha
\end{array}
\right),
\]
for any $C\in \frak c$. Then we obtain  a decomposition
\[\g_\C=V_{\bf 0}(\g_\C)\oplus V_{\alpha_{1}}(\g_\C)\oplus \dots\oplus V_{\alpha_{n}}(\g_\C)
\]
where $\bf 0$ is the trivial character and $\alpha_{1},\dots, \alpha_{n}$ are some non-trivial characters.
We also consider
\[\n_\C=V_{\bf 0}(\n_\C)\oplus V_{\alpha_{1}}(\n_\C)\oplus \dots\oplus V_{\alpha_{n}}(\n_\C)
.\]
Since $\frak c$ is nilpotent, $\frak c$ acts nilpotently on itself via $ad$.
Hence we have ${\frak c}\subset V_{\bf 0}(\g_\C)$ and $V_{\alpha_{i}}(\n_\C)=  V_{\alpha_{i}}(\g_\C)$ by $\g={\frak c}+{\frak n}$ for each $i$ and we get the decomposition
$$
\n_\C=V_{\bf 0}(\n_\C)\oplus V_{\alpha_{1}}(\g_\C)\oplus \dots\oplus V_{\alpha_{n}}(\g_\C).
$$
\begin{defn}
{\rm We say that a solvable Lie algebra   $\g$ is of type (I) if for any $X\in \g$  all the eigenvalues of the adjoint operator $ad_{X}$ are pure imaginary.}
\end{defn}
Note that if we write $\g=\mathfrak{c}+\mathfrak{n}$, where $\mathfrak c$ is an abelian complement of the nilradical $\frak n$, then $\g$ is  of type (I) if and only if the representation $ad:\frak c\to {\rm {Der}} (\n)$ is of type (I).
The following lemma is readily implied by Lemma \ref{kll22}.

\begin{lemma}\label{kll}
If $\g$ is a solvable Lie algebra which is not of type (I). Then
there exists a character $\alpha$ satisfying
$$
\mathop{\rm Re}  (\alpha)\not=0\,,\quad V_{\alpha}(\g_\C)\not=0\,,\mbox{ and }[V_{\alpha}(\g_\C), V_{\bar\alpha}(\g_\C)]=0.
$$
\end{lemma}

\section{Proof of Theorem \ref{spskh}}

In this section we provide a proof of Theorem  \ref{spskh}. The following easy-proof lemma will be useful in the sequel:

\begin{lemma}\label{nilsk}
Let $\frak g$ be a nilpotent Lie algebra and let $\theta$ be a closed $1$-form on $\g$.
Then a $1$-form $\eta$  solves  $d\eta-\eta\wedge\theta =0$ if and only if it is multiple of $\theta$.
\end{lemma}
\begin{proof}
Consider the differential operator $d+\theta\wedge$ acting on $\bigwedge\g^{\ast}$.
Then it is known that  the cohomology of $\bigwedge\g^{\ast}$ with respect to $(d+\theta\wedge)$ is trivial (see \cite{Dix}).
Hence if $\eta\in \bigwedge^{1}\g^{\ast}$ solves  $d\eta-\eta\wedge\theta =0$, then   $\eta$ is  $(d+\theta\wedge)$-exact and so
$\eta\in span_{\R}\langle\theta\rangle $, as required.
\end{proof}

Now we are ready to prove Theorem  \ref{spskh}.
\begin{proof}[Proof of Theorem $\ref{spskh}$]
Firstly we have
\[\frak h_\C= V_{\alpha_{1}}(\frak h_\C)\oplus \dots\oplus V_{\alpha_{n}}(\frak h_\C)
\]
where  $\alpha_{1},\dots, \alpha_{n}$ are some
 characters of $\frak s$.
Therefore $\g_{\C}$ splits as
\[\g_\C={\frak s}_\C \oplus  V_{\alpha_{1}}(\frak h_\C)\oplus \dots\oplus V_{\alpha_{n}}(\frak h_\C).
\]
Then we get
\[[{\frak s}, V_{\alpha_{i}}(\frak h_\C)]\subset V_{\alpha_{i}}(\frak h_\C)
\]
and
\[JV_{\alpha_{i}}(\frak h_\C)\subset V_{\alpha_{i}}(\frak h_\C)
\]
since  $J_{\vert\frak h}\circ \phi(X)=\phi(X)\circ J_{\vert\frak h}$ for any $X\in \frak s$. In view of Lemma \ref{kll}, we may assume that $\alpha_{1}$ satisfies

$$
\mathop{\rm Re}  (\alpha_1)\not=0\,,\quad V_{\alpha_1}(\frak h_\C)\not=0\,,\mbox{ and }[V_{\alpha_1}(\frak h_\C), V_{\bar\alpha_1}(\frak h_\C)]=0
$$
and we can write
\[\bigwedge \g^{\ast}_\C=\bigwedge\left({\frak s}^{\ast}_\C \oplus  V^{\ast}_{\alpha_{1}}(\frak h_\C)\oplus \dots\oplus V^{\ast}_{\alpha_{n}}(\frak h_\C)\right).\]
Then we have
\[d({\frak s}^{\ast}_\C)={\frak s}^{\ast}_\C\wedge {\frak s}^{\ast}_\C,
\]
and by  $[{\frak s}, V_{\alpha_{i}}(\frak h_\C)]\subset V_{\alpha_{i}}(\frak h_\C)$ and $[V_{\alpha_{1}}(\frak h_\C), V_{\bar\alpha_{1}}(\frak h_\C)]=0$,
we obtain
\[
d(V^{\ast}_{\alpha_{i}}(\frak h_\C))\subset  {\frak s}^{\ast}_{\C}\wedge V^{\ast}_{\alpha_{i}}(\frak h_\C)\\
\oplus\bigoplus_{(\alpha_{k},\alpha_{l})\not=(\alpha_{1},\bar\alpha_{1})} V^{\ast}_{\alpha_{k}}(\frak h_\C)\wedge  V^{\ast}_{\alpha_{l}}(\frak h_\C).
\]
Moreover
\[d({\frak s}^{\ast}_\C\wedge V^{\ast}_{\alpha_{i}}(\frak h_\C) \subset {\frak s}^{\ast}_\C\wedge{\frak s}^{\ast}_{\C}\wedge V^{\ast}_{\alpha_{i}}(\frak h_\C)\\
\oplus\bigoplus_{(\alpha_{k},\alpha_{l})\not=(\alpha_{1},\bar\alpha_{1})} {\frak s}^{\ast}_\C\wedge V^{\ast}_{\alpha_{k}}(\frak h_\C)\wedge  V^{\ast}_{\alpha_{l}}(\frak h_\C)
\]
and
\[ d(V^{\ast}_{\alpha_{i}}(\frak h_\C)\wedge V^{\ast}_{\alpha_{j}}(\frak h_\C))\subset  {\frak s}^{\ast}_\C\wedge V^{\ast}_{\alpha_{i}}(\frak h_\C)\wedge V^{\ast}_{\alpha_{j}}(\frak h_\C)+ \frak h^{\ast}_{\C}\wedge\frak h ^{\ast}_{\C}\wedge\frak h^{\ast}_{\C}.
\]
By these relations, we deduce:

\begin{enumerate}
\item[$(\star_{1})$] {\em
the $3$-forms which belongs to the space ${\frak s}^{\ast}_{\C}\wedge  V^{\ast}_{\alpha_{1}}(\frak h_\C)\wedge V^{\ast}_{\bar\alpha_{1}}(\frak h_\C)$ cannot appear in the spaces $d({\frak s}^{\ast}_{\mathbb C}\wedge {\frak s}^{\ast}_{\mathbb C})$,  $d({\frak s}^{\ast}_{\C}\wedge V^{\ast}_{\alpha_{i}}(\frak h_\C)) $ and $ d(V^{\ast}_{\alpha_{i}}(\frak h_\C)\wedge V^{\ast}_{\alpha_{j}}(\frak h_\C))$, excepting $d(V^{\ast}_{\alpha_{1}}(\frak h_\C)\wedge V^{\ast}_{\bar\alpha_{1}}(\frak h_\C))$.}
\end{enumerate}

Consider the operator $d^{c}=J^{-1}dJ$.
Then, assuming $ J\frak s\subset  \frak s$,
 we have
\begin{multline*}
 dd^c({\frak s}^{\ast}_{\C}\wedge V^{\ast}_{\alpha_{i}}(\frak h_\C))\\
\subset
 {\frak s}^{\ast}_{\C}\wedge  {\frak s}^{\ast}_{\C}\wedge{\frak s}^{\ast}_{\C}\wedge V^{\ast}_{\alpha_{i}}(\frak h_\C)
 \bigoplus_{(\alpha_{k},\alpha_{l})\not=(\alpha_{1},\bar\alpha_{1})}{\frak s}^{\ast}_{\C}\wedge  {\frak s}^{\ast}_{\C}\wedge V^{\ast}_{\alpha_{k}}(\frak h_\C)\wedge  V^{\ast}_{\alpha_{l}}(\frak h_\C)
\oplus{\frak s}^{\ast}_{\C}\wedge \frak h^{\ast}_{\C}\wedge\frak h ^{\ast}_{\C}\wedge\frak h^{\ast}_{\C}
\end{multline*}
and
\begin{multline*}
 dd^c(V^{\ast}_{\alpha_{i}}(\frak h_\C)\wedge V^{\ast}_{\alpha_{j}}(\frak h_\C))\\
\subset
{\frak s}^{\ast}_{\C}\wedge{\frak s}^{\ast}_{\C}\wedge V^{\ast}_{\alpha_{i}}(\frak h_\C)\wedge V^{\ast}_{\alpha_{j}}(\frak h_\C)\oplus {\frak s}^{\ast}_{\C}\wedge \frak h^{\ast}_{\C}\wedge\frak h^{\ast}_{\C}\wedge\frak h^{\ast}_{\C}\oplus \frak h^{\ast}_{\C}\wedge\frak h^{\ast}_{\C}\wedge\frak h^{\ast}_{\C}\wedge\frak h^{\ast}_{\C}.
\end{multline*}
By these relations, we have:

\begin{enumerate}
\item[$(\star_{2})$] if  $ J\frak s\subset  \frak s$, then $4$-forms in ${\frak s}^{\ast}_{\C}\wedge {\frak s}^{\ast}_{\C}\wedge V^{\ast}_{\alpha_{1}}(\frak h_\C)\wedge V^{\ast}_{\bar\alpha_{1}}(\frak h_\C)$ do not appear in $dd^c({\frak s}^{\ast}_{\C}\wedge {\frak s}^{\ast}_{\C})$,  $dd^{c}({\frak s}^{\ast}_{\C}\wedge V^{\ast}_{\alpha_{i}}(\frak h_\C)) $ and
$ dd^c(V^{\ast}_{\alpha_{i}}(\frak h_\C)\wedge V^{\ast}_{\alpha_{j}}(\frak h_\C))$, excepting $dd^c(V^{\ast}_{\alpha_{1}}(\frak h_\C)\wedge V^{\ast}_{\bar\alpha_{1}}(\frak h_\C))$.
\end{enumerate}

We are going to prove  the non-existence of taming symplectic (resp. SKT) structures by showing that for any $d$-closed (resp. $dd^{c}$-closed) $2$-form $\Omega$ there exists a non-zero $X\in \g$ such that $\Omega(X,JX)=0$.
We treat the cases $\mathop{\rm Im}  (\alpha_1)\not= 0$ and  $Im(\alpha_1)=0$, separately.

\medskip
\noindent {\em Case $1:\,\,\, \mathop{\rm Im}  (\alpha_1)\not= 0$.}
In this case, we have $V^{\ast}_{\alpha_{1}}(\frak h_\C)\not= V^{\ast}_{\bar\alpha_{1}}(\frak h_\C)$.
The condition $J(V^{\ast}_{\alpha_{1}}(\frak h_\C))\subset V^{\ast}_{\alpha_{1}}(\frak h_\C)$ together the assumption  $\phi\circ J=J\circ \phi$  implies the existence of a basis  $\{e_{1},\dots, e_{p}\}$ of  $V^{\ast}_{\alpha_{1}}(\frak h_\C)$ triangularizing the action of  $\frak s$  on $V^{\ast}_{\alpha_{1}}(\frak h_\C)$ and diagonalizing $J$.
The dual basis $\{e^{1},\dots, e^{p}\}$ satisfies
\[ de^{i}=\delta\wedge e^{i}\;\;\; {\rm mod}\;\;  {\frak s}^{\ast}_{\C}\wedge \langle e^{1},\dots ,e^{i-1}\rangle \oplus \frak h^{\ast}_{\C}\wedge\frak h^{\ast}_{\C}\]
for a closed $1$-form $\delta\in \frak s^{\ast}$.
Each $e^{i}$ could be either
a $(1,0)$-form or a $(0,1)$-form; therefore $\sqrt{-1}e^{i}\wedge \bar e^{i}$ is a real $(1,1)$-form.
Since
\begin{multline*} d(e^{i}\wedge \bar e^{j})=(\delta+\bar\delta)\wedge e^{i}\wedge \bar e^{j}\\
 {\rm mod}\;\; \frak s^{\ast}_{\C}\wedge \langle e^{1},\dots ,e^{i-1}\rangle \wedge \langle\bar e^{j}\rangle +\frak s^{\ast}_{\C}\wedge \langle e^{i}\rangle\wedge \langle \bar e^{1},\dots ,\bar e^{j-1}\rangle + \frak h^{\ast}_{\C}\wedge\frak h^{\ast}_{\C}\wedge \frak h^{\ast}_{\C}
\end{multline*}
condition  $(\star_{1})$, then implies that every closed  $2$-form has no component along $e^{p}\wedge \bar e^{p}$.
Therefore $J$ cannot be tamed by any symplectic form.

Suppose now that $J$ preserves $\frak s$ and $\mathfrak s$ is nilpotent.
Then we get
\begin{multline*} dd^{c}(e^{i}\wedge \bar e^{j})=(dJ(\delta+\bar\delta)-J(\delta+\bar\delta)\wedge(\delta+\bar\delta))\wedge e^{i}\wedge \bar e^{j} \\
{\rm mod}\;\; \frak s^{\ast}_{\C}\wedge  \frak s^{\ast}_{\C}\wedge\langle e^{1},\dots ,e^{i-1}\rangle \wedge \langle \bar e^{1},\dots, \bar e^{j}\rangle +\frak s^{\ast}_{\C}\wedge \frak s^{\ast}_{\C}\wedge \langle e^{1},\dots, e^{i}\rangle\wedge \langle \bar e^{1},\dots ,\bar e^{j-1}\rangle \\
+\frak s^{\ast}_{\C}\wedge  \frak h^{\ast}_{\C}\wedge\frak h^{\ast}_{\C}\wedge \frak h^{\ast}_{\C}+ \frak h^{\ast}_{\C}\wedge  \frak h^{\ast}_{\C}\wedge\frak h^{\ast}_{\C}\wedge \frak h^{\ast}_{\C}.
\end{multline*}
By $\mathop{\rm Re}  (\alpha_{1})\not=0$, we have  $\delta+\bar\delta\not=0$  and $d(\delta+\bar\delta)=0$.
Hence Lemma \ref{nilsk} ensures
$$
dJ(\delta+\bar\delta)-J(\delta+\bar\delta)\wedge(\delta+\bar\delta)\not=0\,.
$$
By $(\star_{2})$, it follows that every $dd^c$-closed $(1,1)$-form has no component along $e^{p}\wedge \bar e^{p}$ and that consequently $J$ doesn't admit any compatible SKT metric.

\medskip
{\em Case $2:\,\, \mathop{\rm Im}  (\alpha_1)=0$.}
In this case, we have $V^{\ast}_{\alpha_{1}}(\frak h_\C)= V^{\ast}_{\bar\alpha_{1}}(\frak h_\C)$.
Since $\alpha_{1}$ is real-valued, we have $V^{\ast}_{\alpha_{1}}(\frak h_\C)=V^{\ast}_{\alpha_{1}}(\frak h)\otimes \C$.
By using $JV^{\ast}_{\alpha_{1}}(\frak h)\subset V^{\ast}_{\alpha_{1}}(\frak h)$ and  $\phi\circ J=J\circ \phi$, we can construct a bais $\{e_{1},\dots , e_{2p}\}$ such that  the action of  $\frak s$ on $V^{\ast}_{\alpha_{1}}(\frak h)$ is trigonalized and $Je^{2k-1}=e^{2k}$ for every $k=1,\dots,p$.
For the dual basis $\{e^{1},\dots, e^{2p}\}$, we have
\[de^{i}=\delta\wedge e^{i} \;\;\; {\rm mod}\;\;  {\frak s}^{\ast}_{\C}\wedge \langle e^{1},\dots ,e^{i-1}\rangle \oplus \frak h^{\ast}_{\C}\wedge\frak h^{\ast}_{\C}\]
for a closed real 1-form $\delta\in \frak s^{\ast}$.
Thus  \[d(e^{i}\wedge e^{j})=2\delta\wedge e^{i}\wedge e^{j}\;\;\; {\rm mod}\;\; {\frak s}^{\ast}_{\C}\wedge \langle e^{1},\dots ,e^{i-1}\rangle\wedge \langle e^{j}\rangle + {\frak s}^{\ast}_{\C}\wedge\langle e^{i}\rangle\wedge \langle e^{1},\dots ,e^{j-1}\rangle+\frak h^{\ast}_{\C}\wedge\frak h^{\ast}_{\C}\wedge\frak h^{\ast}_{\C}.\]
By the condition $\mathop{\rm Re}  (\alpha_{1})\not=0$, we obtain  $\delta\not=0$ and every closed $2$-form $\Omega$ cannot have component along $e^{2p-1}\wedge  e^{2p}$.
Using $(\star_{1})$, we obtain
\[\Omega^{1,1}(e_{2p-1},J(e_{2p-1}))=\Omega^{1,1}(e_{2p-1},e_{2p})=0\]
and $J$ cannot be tamed by any symplectic form, as required.

Suppose now that $J$ preserves $\frak s$ and $\mathfrak s$ is nilpotent.
Then we get
\begin{multline*} dd^{c}(e^{i}\wedge  e^{j})=2(dJ\delta-2J\delta\wedge\delta)\wedge e^{i}\wedge e^{j} \\
{\rm mod}\;\; \frak s^{\ast}_{\C}\wedge  \frak s^{\ast}_{\C}\wedge\langle e^{1},\dots ,e^{i-1}\rangle \wedge \langle  e^{1},\dots,  e^{j}\rangle +\frak s^{\ast}_{\C}\wedge \frak s^{\ast}_{\C}\wedge \langle e^{1},\dots, e^{i}\rangle\wedge \langle  e^{1},\dots , e^{j-1}\rangle \\
+\frak s^{\ast}_{\C}\wedge  \frak h^{\ast}_{\C}\wedge\frak h^{\ast}_{\C}\wedge \frak h^{\ast}_{\C}+ \frak h^{\ast}_{\C}\wedge  \frak h^{\ast}_{\C}\wedge\frak h^{\ast}_{\C}\wedge \frak h^{\ast}_{\C}.
\end{multline*}
By $\mathop{\rm Re}  (\alpha_{1})\not=0$, we have  $\delta\not=0$  and $d\delta=0$.
Hence by Lemma \ref{nilsk}, we have $dJ\delta-2J\delta\wedge\delta\not=0$ and
from $(\star_{1})$ it follows that every $dd^c$-closed $(1,1)$-form has  no component along $e^{2p-1}\wedge  e^{2p}$. Therefore $J$ doesn't admit any compatible SKT metric and  the claim follows.
\end{proof}

As a consequence we get the following

\begin{cor}
Let $G/\Gamma$ be  a   complex parallelizable solvmanifold (i.e. $G$ is a complex Lie group).
Suppose that $G$ is non-nilpotent.
Then $G/\Gamma$ does not admit any SKT-structure.
\end{cor}
\begin{proof}
Let $\frak n$ be the nilradical of the Lie algebra $\frak g$ of $G$.
Take a complex $1$-dimensional subspace $\frak a\subset \frak g$ such that $\frak a\cap \frak n=\{0\}$ and consider
 a subspace $\frak h\subset \frak g$ such that $\g=\frak a\oplus \frak h$ and $\frak n\subset \frak h$.
Since $\frak a$ is a subalgebra of $\g$ and $\frak n\supset [\g,\g]$, $\frak h$ is an ideal of $\g$ and we have $\g=\frak a\ltimes \frak h$.
By $\frak a\cap \frak n=\{0\}$, the action of $\frak a$ on $\frak h$ is non-nilpotent and so the action is  not of type (I).
Hence  the corollary follows from Theorem \ref{spskh}.
\end{proof}

\section{Examples}
In this section we apply Theorem \ref{spskh} to some examples.

\begin{ex}
Let $G=\C\ltimes_{\phi} \C^{2m}$ where
$$
\phi(x+\sqrt{-1}y)(w_{1},w_{2},\dots,w_{2m-1}, w_{2m})
=(e^{a_{1}x}w_{1},e^{-a_{1}x}w_{2},\dots, e^{a_{m}x}w_{2m-1}, e^{-a_{m}x}w_{2m})
$$
for some integers $a_{i}\not=0$.
We denote by $J$ the natural complex structure on $G$.
Then $G$ admits the left-invariant  pseudo-K\"ahler structure
$$
\omega=\sqrt{-1}dz\wedge d\bar z+\sum_{i=1}^{m}(dw_{2i-1}\wedge  d\bar w_{ 2i}+d\bar w_{2i-1}\wedge dw_{2i}).
$$
Moreover $G$ has a co-compact lattice  $\Gamma$ such that  $(G/\Gamma,J)$ satisfies the Hodge symmetry and decomposition (see \cite{KH}). In view of Theorem \ref{spskh}, $(G/\Gamma,J)$ does not admit neither a taming symplectic structure nor an SKT structure.
Moreover by Theorem \ref{almost-Abilian}, $G/\Gamma$ does not admit an invariant complex structure tamed by any symplectic form.
\end{ex}

\begin{ex}[Oeljeklaus-Toma manifolds]
Theorem  \ref{spskh} can be applied to the family of non-K\"ahler complex manifolds constructed by Oeljeklaus and Toma in \cite{OT}. We brightly describe the construction of these manifolds:\\
Let $K$ be a finite extension field of $\Q$ with the degree $s+2t$ for positive integers $s,t$.
Suppose $K$ admits embeddings $\sigma_{1},\dots ,\sigma_{s},\sigma_{s+1},\dots, \sigma_{s+2t}$ into $\C$ such that $\sigma_{1},\dots ,\sigma_{s}$ are real embeddings and $\sigma_{s+1},\dots, \sigma_{s+2t}$ are complex ones satisfying $\sigma_{s+i}=\bar \sigma_{s+i+t}$ for $1\le i\le t$.
We can choose $K$ admitting such embeddings (see \cite{OT}).
Denote ${\mathcal O}_{K}$ the ring of algebraic integers of $K$, ${\mathcal O}_{K}^{\ast}$ the group of units in ${\mathcal O}_{K}$ and
\[{\mathcal O}_{K}^{\ast\, +}=\{a\in {\mathcal O}_{K}^{\ast}: \sigma_{i}>0 \,\, {\rm for \,\,  all}\,\, 1\le i\le s\}.
\]
Define $l:{\mathcal O}_{K}^{\ast\, +}\to \R^{s+t}$ by
\[l(a)=(\log \vert \sigma_{1}(a)\vert,\dots ,\log \vert \sigma_{s}(a)\vert , 2\log \vert \sigma_{s+1}(a)\vert,\dots ,2\log \vert \sigma_{s+t}(a)\vert)
\]
for $a\in {\mathcal O}_{K}^{\ast\, +}$.
Then by Dirichlet's units theorem, $l({\mathcal O}_{K}^{\ast\, +})$ is a lattice in the vector space $L=\{x\in \R^{s+t}\,\,:\,\, \sum_{i=1}^{s+t} x_{i}=0\}$.
Let  $p:L\to \R^{s}$ be  the projection given by the first $s$ coordinate functions.
Then there exists a  subgroup $U$ of ${\mathcal O}_{K}^{\ast\, +}$ of rank $s$  such that $p(l(U))$ is a lattice in $\R^{s}$.
We have the action of $U\ltimes{\mathcal O}_{K}$ on $H^{s}\times \C^{t}$
such that
\begin{multline*}
(a,b)\cdot (x_{1}+\sqrt{-1}y_{1},\dots ,x_{s}+\sqrt{-1}y_{s}, z_{1},\dots ,z_{t})\\
=(\sigma_{1}(a)x_{1}+\sigma_{1}(b)+\sqrt{-1} \sigma_{1}(a)y_{1}, \dots ,\sigma_{s}(a)x_{s}+\sigma_{s}(b)+\sqrt{-1} \sigma_{s}(a)y_{s},\\
 \sigma_{s+1}(a)z_{1}+\sigma_{s+1}(b),\dots ,\sigma_{s+t}(a)z_{t}+\sigma_{s+t}(b)).
\end{multline*}
In \cite{OT} it is proved that the quotient $X(K,U)=H^{s}\times \C^{t}/U\ltimes{\mathcal O}_{K}$ is compact.
We call one of these complex manifolds a  Oeljeklaus-Toma  manifold of type $(s,t)$.

Consider  the Lie group $G=\R^{s}\ltimes_{\phi} (\R^{s}\times \C^{t})$ with
$$
\phi(t_{1},\dots ,t_{s})\\
={\rm diag}(e^{t_{1}},\dots ,e^{t_{s}},e^{\psi_{1}+\sqrt{-1}\varphi_{1}},\dots ,e^{\psi_{t}+\sqrt{-1}\varphi_{t}})
$$
 where $\psi_{k}=\frac{1}{2}\sum_{i=1}^{s}b_{ik}t_{i}$ and $\varphi_{k}=\sum_{i=1}^{s}c_{ik}t_{i}$ for some $b_{ik}, c_{ik}\in \R$.
Let $\g$ be the Lie algebra of $G$.
Then $\bigwedge \g^{\ast}$ is generated by basis  $\{\alpha_{1}, \dots ,\alpha_{s}, \beta_{1}, \dots, \beta_{s}, \gamma_{1},\gamma_{2},\dots ,\gamma_{2t-1}, \gamma_{2t}\}$ satisfying
\[d\alpha_{i}=0,\  d\beta=-\alpha_{i} \wedge \beta_{i},
\]
\[d\gamma_{2i-1}=\bar\psi_{i}\wedge  \gamma_{2i-1}+\bar\varphi_{i} \wedge \gamma_{2i},\,  d\gamma_{2i}=-\bar\varphi _{i}\wedge\gamma_{2i-1}+  \bar\psi_{i}\wedge \gamma_{2i},
\]
where $\bar\psi_{i}=\frac{1}{2}\sum_{i=1}^{s}b_{ik}\alpha_{i}$ and $\bar\varphi_{i} =\sum_{i=1}^{s}c_{ik}\alpha_{i}$.
Consider $w_{i}=\alpha_{i}+\sqrt{-1}\beta_{i}$ for $1\le i\le s$ and  $w_{s+i}=\gamma_{2i-1}+\sqrt{-1} \gamma_{2i}$ as $(1,0)$-forms.
Then $w_{1},\dots ,w_{s+t}$ gives a left-invariant  complex structure $J$ on $G$.
In \cite{KVai}, it is proved that any  Oeljeklaus-Toma manifold of type $(s,t)$ can be regarded as a complex  solvmanifold $(G/\Gamma,J)$.

Consider the $2$-dimensional Lie algebra $\frak r_{2}=span_{\R}\langle A,B\rangle$ such that $[A, B]=B$
and the complex structure $J_{\frak r_{2}}$  on $\frak r_{2}$ defined by the relation $JA=B$.
Then the Lie algebra  $\g$ of $G$  splits as $\g=(\frak r_{2})^{s}\ltimes \C^{t}$ and $J=J_{(\frak r_{2})^{s}}\oplus J_{\C^{t}}$.
 Hence the first part of
 Theorem \ref{spskh} implies that $G/\Gamma$ does not admit  Hermitian-symplectic structures.

On the other hand,  $
(\frak r_{2})^{s}$ is not nilpotent and we cannot apply the second part of  Theorem \ref{spskh} about the existence of SKT structures. Actually, in the case $s=t=1$, the corresponding
Oeljeklaus-Toma manifold $M$  is a $4$-dimensional solvmanifold and by the unimodularity any invariant $3$-form is closed forcing $M$ to be SKT. For $s\neq 1 $ things work differently:
\begin{prop}
Let $s\geq 2$. Then every Oeljeklaus-Toma manifold of type $(s,1)$ does not admit a SKT structure.
\end{prop}
\begin{proof}
In case $t=1$, we have $G=\R^{s}\ltimes _{\phi}(\R^{s}\times \C) $ where
$$
\phi(t_{1},\dots ,t_{s})={\rm diag}(e^{t_{1}},\dots ,e^{t_{s}},e^{-\frac{1}{2}(t_{1}+\dots+ t_{s})+\sqrt{-1}\varphi_{1}}).
$$
Then  $\bigwedge \g^{\ast}$ is generated by a basis  $\{\alpha_{1}, \dots ,\alpha_{s}, \beta_{1}, \dots , \beta_{s}, \gamma_{1},\gamma_{2}\}$ satisfying
\[d\alpha_{i}=0,\  d\beta=-\alpha_{i} \wedge \beta_{i},
\]
\[d\gamma_{1}=\frac{1}{2}\theta\wedge \gamma_{1}+\bar\varphi_{1} \wedge \gamma_{2},\,  d\gamma_{2}=-\bar\varphi_{1} \wedge\gamma_{1}+  \frac{1}{2}\theta\wedge \gamma_{2},
\]
where $\theta=\alpha_{1}+\dots +\alpha_{s}$ (see \cite{KVai}).
Let us consider the left-invariant $(1,0)$ coframe
 $$
\begin{aligned}
w_{i}&=\alpha_{i}+\sqrt{-1}\beta_{i}\,, \mbox{ for } 1\le i\le s\\
w_{s+1}&=\gamma_{1}+\sqrt{-1} \gamma_{2}\,.
\end{aligned}
$$
This coframe induces a global left-invariant coframe on the corresponding Oeljeklaus-Toma manifold $M=G/\Gamma$.
We have
\[dd^{c}(w_{s+1}\wedge  \bar w_{s+1})
=(dJ\theta-J\theta\wedge\theta)\wedge w_{s+1} \wedge \bar w_{s+1}\]
and
\[dJ\theta-J\theta\wedge\theta=-(\alpha_{1}\wedge\beta_{1}+\dots+\alpha_{s}\wedge\beta_{s})-(\beta_{1}+\dots\beta_{s})\wedge(\alpha_{1}+\dots+\alpha_{s})
\not=0.\]
It follows that if $\Omega$ is a $(1,1)$-form satisfying  $dd^c\Omega=0$, then $\Omega$ has no component along $ w_{s+1}\wedge \wedge \bar w_{s+1}$. This implies that every $dd^c$-closed $(1,1)$-form on $M$ is degenerate, as require.
Hence the proposition follows.
\end{proof}
\end{ex}

\begin{ex} In \cite{Yam} it was introduced the following Lie algebra admitting  pseudo-K\"ahler structures:\\
 Let $\g= {\mbox {span}}_{\R} \langle A_{i},  W_{i},X_{j} , Y_{j}, Z_{j},X_{j}^{\prime}, Y_{j}^{\prime}, Z_{j}^{\prime}\rangle_{i=1,2, j=1,2,3,4}$ where
\[[A_1,A_2] = W_1,\]
\[[X_1, Y_1] = Z_1,\;\;\;	[X_3, Y_3] = Z_3,\]
\[ [A_1, X_1] = t_0 X_1, \; [A_1, X_2]= t_{0} X_2,\;
[A_{1},X_{3}=-t_{0}X_{3},\; [A_{1},X_{4}]=-t_{0}X_{4},
\]
\[[A_{1},Y_{1}]=-2t_{0}Y_{1},\; [A_{1}, Y_{2}]=-2t_{0}Y_{2},\; [A_{1},Y_{3}]=2t_{0}Y_{3},\;[A_{1},Y_{4}]=2t_{0}Y_{4},\]
\[[A_{1},Z_{1}]=-t_{0}Z_{1},\; [A_{1},Z_{2}]=-t_{0}Z_{2},\; [A_{1},Z_{3}]=t_{0}Z_{3},\; [A_{1},Z_{4}]=t_{0}Z_{4},
\]
\[[X_{2}, Y_{1}]=Z_{2},\; [X_{4},Y_{3}]=Z_{4}],\]
\[[X^\prime_{1}, Y^{\prime}_1] = Z^{\prime}_1,\;\;\;	[X^{\prime}_3, Y^{\prime}_3] = Z^{\prime}_3,\]
\[ [A_{2}, X^{\prime}_1] = t_0 X^{\prime}_1, \; [A_2, X^{\prime}_2]= t_{0} X^{\prime}_2,\;
[A_{2},X^{\prime}_{3}=-t_{0}X^{\prime}_{3},\; [A_{2},X^{\prime}_{4}]=-t_{0}X^{\prime}_{4},
\]
\[[A_{2},Y^{\prime}_{1}]=-2t_{0}Y^{\prime}_{1},\; [A_{2}, Y^{\prime}_{2}]=-2t_{0}Y^{\prime}_{2},\; [A_{2},Y^{\prime}_{3}]=2t_{0}Y^{\prime}_{3},\;[A_{2},Y^{\prime}_{4}]=2t_{0}Y^{\prime}_{4},\]
\[[A_{2},Z^{\prime}_{1}]=-t_{0}Z^{\prime}_{1},\; [A_{2},Z^{\prime}_{2}]=-t_{0}Z^{\prime}_{2},\; [A_{2},Z^{\prime}_{3}]=t_{0}Z^{\prime}_{3},\; [A_{2},Z^{\prime}_{4}]=t_{0}Z^{\prime}_{4},
\]
\[[X^{\prime}_{2}, Y^{\prime}_{1}]=Z^{\prime}_{2},\; [X^{\prime}_{4},Y^{\prime}_{3}]=Z^{\prime}_{4}\]
and the other brackets vanish.
Then the simply connected solvable Lie group $G$ corresponding to $\g$ has a lattice (see \cite{Yam}).
We can write $\g= {\mbox {span}}_{\R} \langle A_{i},W_{i}\rangle_{i=1,2}\ltimes  {\mbox {span}}_{\R}\langle X_{j} , Y_{j}, Z_{j},X_{j}^{\prime}, Y_{j}^{\prime}, Z_{j}^{\prime}\rangle_{j=1,2,3,4}$ and
$G$ has the left-invariant complex structure $J$ defined as
\[ JA_{1}=A_{2},\; JW_{1}=W_{2},\]
\[JX_{1}=X_{2},\; JY_{1}=Y_{2},\; JZ_{1}=Z_{2},\; JX_{3}=X_{4},\; JY_{3}=Y_{4},\; JZ_{3}=Z_{4},\]
\[JX^{\prime}_{1}=X^{\prime}_{2},\; JY^{\prime}_{1}=Y^{\prime}_{2},\; JZ^{\prime}_{1}=Z^{\prime}_{2},\; JX^{\prime}_{3}=X^{\prime}_{4},\; JY_{3}=Y_{4},\; JZ^{\prime}_{3}=Z^{\prime}_{4}.\]
In view of Theorem \ref{spskh}, $G/\Gamma$ does not admit any SKT structure compatible with $J$.
\end{ex}

\section{Proof of Theorem \ref{th4.2}}
The proof of Theorem \ref{th4.2} is mainly based on the following proposition which is interesting in its own.
\begin{prop}\label{Lieta}
Let $G$ be a simply-connected solvable Lie group whose Lie algebra $\g$ is not of type (I).
Let $J$ be a left-invariant complex structure on $G$ satisfying
$$
ad_{C}\circ J=J\circ ad_{C}
$$
for every  $C$ belonging to a nilpotent complement $\frak c$ of the nilradical of $\g$. Then  $G$ does not admit  any  left-invariant symplectic form  taming $J$.
\end{prop}
\begin{proof}
By Section \ref{sol}, we have
\[\g_\C=V_{\bf 0}(\g_\C)\oplus V_{\alpha_{1}}(\g_\C)\oplus \dots\oplus V_{\alpha_{n}}(\g_\C)
\]
where $\bf 0$ is the trivial character and $\alpha_{1},\dots, \alpha_{n}$ are some non-trivial characters.
Take a subspace $\frak a\subset \frak c$ such that $\g=\frak a\oplus \n$.
Then we have
\[\g_\C={\frak a}_\C \oplus V_{\bf 0}(\n_\C)\oplus V_{\alpha_{1}}(\g_\C)\oplus \dots\oplus V_{\alpha_{n}}(\g_\C).
\]
So we obtain
\[[{\frak a}_{\C}, V_{\bf 0}(\n_\C)]\subset V_{\bf 0}(\n_\C), \quad \quad
[{\frak a}_{\C}, V_{\alpha_{i}}(\g_\C)]\subset V_{\alpha_{i}}(\g_\C)
\]
and
\[J V_{\alpha_{i}}(\g_\C)\subset V_{\alpha_{i}}(\g_\C).
\]
By Lemma \ref{kll}, we may assume that  $\alpha_{1}$ satisfies
$$
\mathop{\rm Re}   (\alpha_{1})\not=0\,,\quad V_{\alpha_{1}}(\g_\C)\not=0\, \mbox{ and }[V_{\alpha_{1}}(\g_\C), V_{\bar\alpha_{1}}(\g_\C)]=0\,.
$$
Consider the natural splitting
\[\bigwedge \g^{\ast}_\C=\bigwedge\left({\frak a}^{\ast}_\C \oplus V^{\ast}_{\bf 0}(\n_\C)\oplus V^{\ast}_{\alpha_{1}}(\g_\C)\oplus \dots\oplus V^{\ast}_{\alpha_{n}}(\g_\C)\right).\]
Then we have
\[d({\frak a}^{\ast}_\C)=0
\]
and, by taking into account  $[{\frak a}, V_{\bf 0}(\n_\C)]\subset V_{\bf 0}(\n_\C)$, $[{\frak a}, V_{\alpha_{i}}(\g_\C)]\subset V_{\alpha_{i}}(\g_\C)$ and $[V_{\alpha_{1}}(\g_\C), V_{\bar\alpha_{1}}(\g_\C)]=0$,
we get
\begin{multline*}d(V^{\ast}_{\bf 0}(\n_\C))\subset {\frak a}^{\ast}_{\C}\wedge{\frak a}^{\ast}_{\C}\oplus {\frak a}^{\ast}_{\C}\wedge V^{\ast}_{\bf 0}(\n_\C)\\
\oplus\bigoplus_{(\alpha_{k},\alpha_{l})\not=(\alpha_{1},\bar\alpha_{1})} V^{\ast}_{\alpha_{k}}(\g_\C)\wedge  V^{\ast}_{\alpha_{l}}(\g_\C)\oplus\bigoplus V^{\ast}_{\alpha_{m}}(\g_\C)\wedge V^{\ast}_{\bf 0}(\n_\C),
\end{multline*}
and
\begin{multline*}d(V^{\ast}_{\alpha_{i}}(\g_\C))\subset {\frak a}^{\ast}_{\C}\wedge{\frak a}^{\ast}_{\C}\oplus {\frak a}^{\ast}_{\C}\wedge V^{\ast}_{\alpha_{i}}(\g_\C)\\
\oplus\bigoplus_{(\alpha_{k},\alpha_{l})\not=(\alpha_{1},\bar\alpha_{1})} V^{\ast}_{\alpha_{k}}(\g_\C)\wedge  V^{\ast}_{\alpha_{l}}(\g_\C)\oplus\bigoplus V^{\ast}_{\beta_{m}}(\g_\C)\wedge V^{\ast}_{\bf 0}(\n_\C).
\end{multline*}
Hence we have
\[d({\frak a}^{\ast}_{\C}\wedge {\frak a}^{\ast}_{\C})=0
,\]
and
\begin{multline*}
 d({\frak a}^{\ast}_{\C}\wedge V^{\ast}_{\alpha_{i}}(\g_\C))\subset {\frak a}^{\ast}_{\C}\wedge {\frak a}^{\ast}_{\C}\wedge {\frak a}^{\ast}_{\C}\oplus {\frak a}^{\ast}_{\C}\wedge {\frak a}^{\ast}_{\C}\wedge V^{\ast}_{\alpha_{i}}(\g_\C)\\
 \oplus \bigoplus_{(\alpha_{k},\alpha_{l})\not=(\alpha_{1},\bar\alpha_{1})}  {\frak a}^{\ast}_{\C}\wedge V^{\ast}_{\alpha_{k}}(\g_\C)\wedge  V^{\ast}_{\alpha_{l}}(\g_\C)\oplus\bigoplus {\frak a}^{\ast}_{\C}\wedge V^{\ast}_{\alpha_{m}}(\g_\C)\wedge V^{\ast}_{\bf 0}(\n_\C)
\end{multline*}
and
\begin{multline*}
 d(V^{\ast}_{\alpha_{i}}(\g_\C)\wedge V^{\ast}_{\alpha_{j}}(\g_\C))\subset {\frak a}^{\ast}_{\C}\wedge {\frak a}^{\ast}_{\C}\wedge V^{\ast}_{\alpha_{i}}(\g_\C)\oplus {\frak a}^{\ast}_{\C}\wedge {\frak a}^{\ast}_{\C}\wedge V^{\ast}_{\alpha_{j}}(\g_\C)
\\
\oplus {\frak a}^{\ast}_{\C}\wedge V^{\ast}_{\alpha_{i}}(\g_\C)\wedge V^{\ast}_{\alpha_{j}}(\g_\C)\oplus \n^{\ast}_{\C}\wedge\n^{\ast}_{\C}\wedge\n^{\ast}_{\C}.
\end{multline*}
Combining these relations we have:

\begin{enumerate}
\item[$(\diamond)$]  $3$-forms in ${\frak a}^{\ast}_{\C}\wedge V^{\ast}_{\alpha_{1}}(\g_\C)\wedge V^{\ast}_{\bar\alpha_{1}}(\g_\C)$ do not appear in $d({\frak a}^{\ast}_{\C}\wedge {\frak a}^{\ast}_{\C})$,  $d({\frak a}^{\ast}_{\C}\wedge V^{\ast}_{\alpha_{i}}(\g_\C)) $ and $ d(V^{\ast}_{\alpha_{i}}(\g_\C)\wedge V^{\ast}_{\alpha_{j}}(\g_\C))$, excepting $d(V^{\ast}_{\alpha_{1}}(\g_\C)\wedge V^{\ast}_{\bar\alpha_{1}}(\g_\C))$.
\end{enumerate}

The non-existence of taming symplectic structures will be obtained by showing that for any $d$-closed  $2$-form $\Omega$ there exists a non-trivial $X\in \g$ such that $\Omega(X,JX)=0$.
From now on, we distinguishe the case where $\mathop{\rm Im}  (\alpha_1)\not= 0$ from the case $\mathop{\rm Im}  (\alpha_1)=0$.

\medskip
\noindent {\em Case $1:\,\,\, \mathop{\rm Im}  (\alpha_1)\not= 0$.}
In this case we have $V^{\ast}_{\alpha_{1}}(\g_\C)\not= V^{\ast}_{\alpha_{1}}(\g_\C)$.
Since $J(V^{\ast}_{\alpha_{1}}(\g_\C))\subset V^{\ast}_{\alpha_{1}}(\g_\C)$, there exists a basis $\{e_{1},\dots, e_{p}\}$ such that the action of $\frak c$  onto $V^{\ast}_{\alpha_{1}}(\g\otimes \C)$ is trigonalized and $J$ is diagonalized.
The dual basis $\{e^{1},\dots, e^{p}\}$ satisfies
\[ 
de^{i}=\delta\wedge e^{i}\;\;\; {\rm mod}\;\;  {\frak a}^{\ast}_{\C}\wedge \langle e^{1},\dots ,e^{i-1}\rangle \oplus \n^{\ast}_{\C}\wedge\n^{\ast}_{\C}\]
for a complex closed form $\delta\in {\frak a}^{\ast}_{\C}$.
 Each $e^{i}$ is either a  $(1,0)$ or a  $(0,1)$-form and so $\sqrt{-1}e^{i}\wedge \bar e^{i}$ is a real $(1,1)$-form.
Therefore
\begin{multline*} 
d(e^{i}\wedge \bar e^{j})=(\delta+\bar\delta)\wedge e^{i}\wedge e^{j}\\
 {\rm mod}\;\; \frak a^{\ast}_{\C}\wedge \langle e^{1},\dots ,e^{i-1}\rangle \wedge \langle\bar e^{j}\rangle +\frak a^{\ast}_{\C}\wedge \langle e^{i}\rangle\wedge \langle \bar e^{1},\dots ,\bar e^{j-1}\rangle + \n^{\ast}_{\C}\wedge\n^{\ast}_{\C}\wedge \n^{\ast}_{\C}.
\end{multline*}
By $\mathop{\rm Re}  (\alpha_{1})\not=0$, we have $\delta+\bar\delta\not=0$.
Hence $(\diamond)$ implies that every closed  $2$-form $\Omega$ has no component along $e^{p}\wedge \bar e^{p}$.
Hence
\[\Omega^{1,1}(e_{p}+\bar e_{p},J(e_{p}+\bar e_{p}))=\Omega^{1,1}(e_{p}+\bar e_{p},\sqrt{-1}(e_{p}-\bar e_{p}))=0\]
and  $J$ cannot be tamed by any symplectic form.

\medskip
\noindent {\em Case $2:\,\,\, \mathop{\rm Im}  (\alpha_1)= 0$.}
In this case we have $V^{\ast}_{\alpha_{1}}(\g_\C)= V^{\ast}_{\alpha_{1}}(\g_\C)$.
Since $\alpha_{1}$ is real-valued, we have $V^{\ast}_{\alpha_{1}}(\g_\C)=V^{\ast}_{\alpha_{1}}(\g)\otimes \C$.
Since $JV^{\ast}_{\alpha_{1}}(\g)\subset V^{\ast}_{\alpha_{1}}(\g)$ and  $ad_{C}\circ J=J\circ ad_{C}$ for any $C\in\frak c$ there exists a basis $\{e_{1},\dots, e_{2p}\}$ such that  the action of  $\frak c$ on $V^{\ast}_{\alpha_{1}}(\g)$ is trigonalized and $Je^{2k-1}=e^{2k}$ for each $k$.
Let $\{e^{1},\dots, e^{2p}\}$ be the dual basis. Then
\[de^{i}=\delta\wedge e^{i} \;\;\; {\rm mod}\;\;  {\frak a}^{\ast}_{\C}\wedge \langle e^{1},\dots ,e^{i-1}\rangle \oplus \n^{\ast}_{\C}\wedge\n^{\ast}_{\C}\]
for a real closed form $\delta\in \bigwedge \frak a^{\ast}$.
Hence  we have
\[d(e^{i}\wedge e^{j})=2\delta\wedge e^{i}\wedge e^{j}\;\;\; {\rm mod}\;\; {\frak a}^{\ast}_{\C}\wedge \langle e^{1},\dots ,e^{i-1}\rangle\wedge \langle e^{j}\rangle + {\frak a}^{\ast}_{\C}\wedge\langle e^{i}\rangle\wedge \langle e^{1},\dots ,e^{j-1}\rangle+\n^{\ast}_{\C}\wedge\n^{\ast}_{\C}\wedge\n^{\ast}_{\C}.\]
By $\mathop{\rm Re}  (\alpha_{1})\not=0$, we have $\delta\not=0$.
Hence by $(\diamond)$,  every closed $2$-form $\Omega$ has no component along $e^{2p-1}\wedge e^{2p}$.
Hence we have
\[\Omega^{1,1}(e_{2p-1},Je_{2p-1})=\Omega^{1,1}(e_{2p-1},e_{2p})=0.\]
and  $J$ cannot be tamed by any symplectic form, as required.
\end{proof}

Now we are ready to prove Theorem \ref{th4.2}.

\begin{proof}[Proof of Theorem $\ref{th4.2}$]
 In view of \cite{EFV} the existence  of a symplectic form taming $J$  implies the existence of an {\em invariant} symplectic form taming $J$. Hence it is enough to prove that there are no invariant symplectic forms taming $J$.
By Proposition \ref{Lieta}, the Lie algebra $\g$ is not of type (I).
Given a  nilpotent complement $\frak c\subset \g$,
 we define the diagonal representation
\[ad_{s}:\g=\frak c +\frak n\ni C+X\mapsto (ad_{C})_{s}\in D(\g).\]
Consider the extension $Ad_{s}:G\to Aut(\g)$.
Then the Zariski-closure $T={\mathcal A}(Ad_{s}(G))$ in $Aut(\g)$ is a maximal torus of the Zariski-closure ${\mathcal A}(Ad(G))$ (see \cite{Kas} and \cite{CFK}).
It is known that there exists a simply-connected nilpotent Lie group $U_{G}$, called the {\em nilshadow} of $G$, which is independent on the  choice of $T$ and satisfies $T\ltimes G=T\ltimes U_{G}$.
From \cite{CFK} it follows that if $J$ is a left-invariant complex structure on $G$ satisfying $J\circ Ad_{s}=Ad_{s}\circ J$, then $U_{G}$ inherits
a left-invariant complex structure $\tilde J$  such that $(U_{G},\tilde J)$ is bi-holomorphic to $(G,J)$.
Now every lattice of $G$ induces a discrete subgroup $\Gamma$ in $T\ltimes U_{G}$ such that  $\tilde\Gamma=U_{G}\cap \Gamma$ is a lattice of $U_{G}$ and has finite index in $\Gamma$ (see \cite[Chapter V-5]{Aus}).
There follows that   $(G/\tilde \Gamma,J)$ is bi-holomorphic to $(U_{G}/\tilde\Gamma,\tilde J)$. Hence $U_{G}/\tilde\Gamma$ is a finite covering of a Hermitian-symplectic manifold and, consequently, it inherits an invariant symplectic form $\tilde \Omega$ taming $\tilde J$.
By the main result of  \cite{EFV} it follows that $U_{G}/\tilde\Gamma$ is a torus.
Hence $(G/\Gamma,J)$ is a finite quotient of a complex torus $U_{G}/\tilde\Gamma$ by a finite group of holomorphic automorphisms and
by  \cite{BR},  $(G/\Gamma,J)$ admits a K\"ahler metric.
\end{proof}

\section{Abelian complex structures}\label{abilian}
In this section we consider abelian complex structures providing a proof of Theorem \ref{thABILIAN}.

\medskip
Theorem \ref{thABILIAN} is mainly motivated by the research in \cite{ABD}  where it is showed that a Lie group  with a left-invariant abelian complex structure admits a compatible left-invariant K\"ahler
structure if and only if it is a direct product of several copies of the real hyperbolic plane by an Euclidean factor.
Moreover, from \cite[Lemma 2.1]{ABD}  it follows that a Lie algebra  $\frak g$ with an abelian complex structure $J$ has the following properties:
\begin{enumerate}

\item[1.]  the center  $\xi (\frak g)$  of  $\frak g$  is $J$-invariant;

\vspace{0.1cm}
\item[2.] for any $X \in \frak g$,  $ad_{JX} =  - ad_X  J$;

\vspace{0.1cm}
\item[3.] the commutator $\frak g^1 = [\frak g, \frak g]$   is abelian or, equivalently, $\frak g$  is 2-step solvable;

\vspace{0.1cm}
\item[4.]  $J \frak g^1$ is an abelian subalgebra of $\frak g$;

\vspace{0.1cm}
\item[5.]  $ \frak g^1 \cap J \frak g^1$ is contained in the center of the subalgebra  $\frak g^1 + J \frak g^1$.

\end{enumerate}

Our Theorem \ref{thABILIAN} can be easily deduced in dimension 4 and 6 by using the classification of Lie algebras admitting an abelian complex structure.
Indeed, by the classifications  in dimensions $4$  (\cite{Sn}) and $6$ (\cite{ABD1})  we know that if  $(\g,J)$ is a unimodular Lie algebra  with an
 abelian complex structure, then  the existence of  a symplectic form  taming $J$ implies that $\frak g$ is abelian.  In dimension $4$ this fact follows from \cite{EF}. In dimension $6$ we use that the only  unimodular (non-nilpotent) Lie algebra admitting an abelian complex structure is holomorphically isomorphic to $( \frak s_{(-1,0)}, J)$, where  $\frak s_{(-1,0)}$  is the  solvable Lie algebra with Lie brackets
$$
\begin{array}{l}
[f_1, e_1] = [f_2, e_2] = e_1, \quad [f_1, e_2] = - [f_2, e_1] = e_2,\\[3pt]
[f_1, e_3] = [f_2, e_4] = - e_3, \quad [f_1, e_4] = - [f_2, e_3] = - e_4
\end{array}
$$
and  the abelian complex structure $J$  is given by
$$
J f_1 = f_2,  \, J e_1 = e_2,  \, J e_3 = e_4.
$$
This Lie algebra   has  nilradical   $\frak n =  {\rm {span}}_{\R}  \langle e_1, e_2, e_3, e_4\rangle $ and $ad_c \circ J = J \circ ad_c$,  for every  $c \in {\frak c} = \langle f_1, f_2\rangle $. Since $\frak c$   is an abelian complement of $\frak n$,    Theorem \ref{Lieta} implies  that $( \frak s_{(-1,0)}, J)$ does not admit any symplectic form taming $J$.

Theorem \ref{thABILIAN} follows from the following

\begin{prop} Let  $(\frak g,J)$ be a unimodular Lie algebra with an abelian complex structure. Assume that  there exists a symplectic form $\Omega$ on $\g$ taming $J$. Then $\frak g$ is abelian.
\end{prop}

\begin{proof}
Since the pair $(J, \Omega)$ induces a Hermitian symplectic structure on every $J$-invariant subalgebra of $\g$ and $\g^1$ and $J\g^1$ are both abelian Lie subalgebras  of  $\frak g$, it is quite natural
to work  with $\frak g^1 + J \frak g^1$. We have the  following two cases which  we will  treat separately:

\begin{enumerate}
\item[\em Case A]:  $\frak g^1 + J \frak g^1 = \frak g$

\item[\em Case B]: $\frak g^1 + J \frak g^1  \neq \frak g$.
\end{enumerate}

In the Case A we necessary have $\frak g^1 \cap J \frak g^1 = \{ 0 \}$,  since otherwise by using   that  $ \frak g^1 \cap J \frak g^1 \subseteq \xi (\frak g)$,  it should exist  a non-zero $X \in    J \xi (\frak g) \cap \frak g^1 $, but this contradicts Lemma 3.1 in \cite{EFV}. Therefore
$$
\frak g = \frak g^1 \oplus J \frak g^1,
$$
or
equivalently $\frak g$ is an abelian double product. As a consequence of   Corollary 3.3 in \cite{ABD}
the Lie bracket in  $\frak g$  induces a structure of commutative and associative algebra on  $\frak g^1$  given by
$$
X \cdot  Y = [JX,  Y].
$$
Let  $\mathcal A:=(\frak g^1,  \, \cdot)$. Then  $\mathcal A^2  = \mathcal A$ and  $(\frak g, J)$  is holomorphically
isomorphic to
${\rm {aff}} ({\mathcal A}) = {\mathcal A} \oplus {\mathcal A}$  with the standard complex structure
$$
J (X, Y) = (Y, - X).
$$
Note that in general  the Lie bracket on the affine Lie algebra ${\rm {aff}} ({\mathcal A})$ associated to a commutative associative algebra $(\mathcal A, \cdot)$  is given by
$$
[ (x, y), (x', y')] = (0, x \cdot  y' - x' \cdot y),
$$
for every $(x, y), (x', y') \in {\rm {aff}} ({\mathcal A})$.
Moreover, $ {\rm {aff}} (\mathcal A)$ is nilpotent if and only if $ \mathcal A$ is nilpotent as associative algebra.
We are going to show now that when ${\rm {aff}}({\mathcal A})$ is unimodular and it  is endowed with a symplectic form taming $J$, then  the  Lie algebra ${\rm {aff}} ({\mathcal A})$ is forced to be abelian.
Since we know that this is true in dimension $4$ and $6$ we can prove the assertion by induction on the dimension of $\mathcal A$.
We  may assume that  $\mathcal A$  is not a direct
sum of proper non-trivial ideals, since otherwise if $\mathcal A = \mathcal A_1 \oplus \cdots \mathcal \oplus  A_k$, then ${\rm {aff}} ({\mathcal A}) = {\rm {aff}} ({\mathcal A}_1) \oplus  \cdots \oplus {\rm {aff}} ({\mathcal A}_k)$ and by induction we obtain that any $ {\rm {aff}} ({\mathcal A}_k)$ is abelian.
Since $\mathcal A$ is a commutative associative algebra over $\R$, by applying Lemma 3.1 in \cite{BDeG},  we get that $\mathcal A$  is either
\begin{enumerate}
 \item[ (i)] nilpotent, or
\item [(ii)] equal to $\tilde  {\mathcal B}  =  \mathcal  B   \oplus \R  \langle1 \rangle $ for a nilpotent commutative associative algebra $\mathcal B$,  where by $1$ we denote the unit of $\mathcal A$
or
\item [(iii)] equal to  $\C \oplus  \mathcal R$, where $\mathcal R$ is the radical of $\mathcal A$.
\end{enumerate}
Since ${\rm {aff}} (\C)$ is not unimodular then we can exclude the case (iii).
Moreover, in the case (ii) ${\rm {aff}} (\mathcal A)$ cannot be unimodular, since
$$
[(1, 0), (x', y')] = (0, y'),
$$
for every $(x', y')  \in  {\rm {aff}} ({\mathcal A})$.  In particular, $[(1, 0), (0, 1)] = (0, 1)$ and then $trace (ad_{(1,0)}) \neq 0$.
We conclude then that the Lie algebra ${\rm {aff}} ({\mathcal A})$ has to be nilpotent and by \cite{EFV}  $ {\rm {aff}} ({\mathcal A})$ has to be abelian, since it is Hermitian-symplectic.

Let us consider now the Case B in which  $\frak g^1 + J \frak g^1$ is a proper ideal of $\g$.
By induction on the dimension we may assume that $\frak g^1 + J \frak g^1$ is abelian.
Fix an arbitrary  $J$-invariant  complement $\frak h$ of $\frak g^1 + J \frak g^1$. We show that $[\mathfrak h,\mathfrak g^1+ J\mathfrak g^1]=0$ proving in this way that $\mathfrak g$ is nilpotent. Fix $X\in \mathfrak h$ and consider the following two bilinear forms on $\g^1+J\g^1$
$$
B_X(Y,Z):=\Omega([X,Y],Z)\,,\quad B'_X(Y,Z):=\Omega([JX,Y],Z)\,.
$$
Since $\Omega$  is closed and $\g^1+J\g^1$ is abelian,  the two bilinear forms  $B_X$ and $B'_X$ are both symmetric. On the other hand the abelian condition on $J$ ensures that
$$
B'_X(Y,Z)=-B_X(JY,Z),
$$
for every $Y, Z \in \g^1+J\g^1$. Thus
$$
\begin{array}{lcl}
B_X (JY, JZ) &=& \Omega ([X, JY], JZ) = - \Omega([JX, Y], JZ) \\[3pt]
&=& - B'_X(Y, JZ) = - B'_X (JZ, Y) = - \Omega([JX, JZ], Y) \\[3pt]
&= & - \Omega ([X, Z], Y) = - B_X(Y, Z),
\end{array}
$$
for every $Y, Z \in \g^1+J\g^1$ or, equivalently,
$$
\Omega([X,JY],JZ)=-\Omega([X,Y],Z)\,, \quad \forall \,\, Y, Z\,\, \in \g^1+J\g^1.
$$
In particular
$$
\Omega([X,JY],J[X,JY])=\Omega([X,Y],[JX,Y])\,, \quad  \forall\,\,Y, Z \in\g^1+J\g^1.
$$
We finally show that $\Omega([X,Y],[JX,Y])=0$ obtaining in this way $[X,JY]=0$.

Indeed,
$$
\begin{array}{lcl}
\Omega([X,Y],[JX,Y]) &= &\Omega([X,[JX,Y]],Y)\\[3pt]
&=& -\Omega([Y,[X,JX]],Y)-\Omega([JX,[Y,X]],Y)\\[3pt]
&= &-\Omega([X,Y],[JX,Y])\,,
\end{array}
$$
which implies $\Omega([X,Y],[JX,Y])=0$, as required. Therefore $[\mathfrak h, \g^1+J\g^1 ]=0$ and $\g$ is nilpotent. Finally
Theorem 1.3 in \cite{EFV} implies that $\frak g$ is abelian, as required.
\end{proof}

\section{Almost-abelian  solvmanifolds}

By \cite{Ov} a  $4$-dimensional  unimodular Hermitian symplectic Lie algebra $\frak g$ is  K\"ahler and it is isomorphic to the almost abelian Lie algebra  $\tau \tau'_{3,0}$ with structure equations
$$
[e_1, e_2] = - e_3, \quad [e_1, e_3] = e_2.
$$
Note that indeed  a   $4$-dimensional  unimodular (non abelian) Lie algebra  $\frak g$  is symplectic if and only if it  is isomorphic either  to  the $3$-step $4$-dimensional  nilpotent Lie algebra   or
to a direct product of  $\R$ with a $3$-dimensional unimodular solvable Lie algebra.

The proof of Theorem \ref{almost-Abilian} is implied by the two subsequent propositions. The first one implies the statement of Theorem \ref{almost-Abilian} when $\g$ is not of type (I).

\begin{prop} \label{almostabrealtype}  Let $J$ be a  complex structure    on a unimodular    almost abelian (non-abelian)  Lie algebra $\g$.
If $\g$ is  not of type (I), then $\g$ does not admit a  symplectic structure taming $J$.
\end{prop}

\begin{proof} Let $\frak n$ be the nilradical of $\g$. Since $\g$ is almost abelian we have that $\frak n$ has codimension $1$ and $\frak n$ is abelian. Let $\Omega$ be a symplectic form taming $J$ and $g$ the associated $J$-Hermitian metric.
We recall that this metric is defined as the Hermitian metric induced by $(1,1)$-component $\Omega^{1,1}$ of $\Omega$.
With respect to the Hermitian metric $g$ we have the orthogonal decomposition
$$
\g = \frak n  \, \oplus {\mbox{span}}_{\R}  \langle X \rangle.
$$
Since $JX$ is orthogonal to $X$, $JX$ belongs to $\frak g^1$ and thus $JX \in \frak n$. By the unimodularity of $\g$, we get that  $[X, JX]$ belongs to the  the orthogonal complement  of ${\mbox{span}}_{\R} \langle  X , JX\rangle $ with respect to $g$, i.e. to the $J$-invariant abelian Lie subalgebra
$$
\frak h =  {\mbox{span}}_{\R}  \langle  X , JX\rangle ^{\perp}.
$$
Since $\frak n$ is abelian,  by using the integrability of $J$  we obtain
$$
ad_X (JY) = J ad_X (Y),
$$
for every $Y \in \frak h$. We can show that $\frak h$ is $ad_X$-invariant.  Indeed, we know that
$$
g ([X, Y], X) = 0, \mbox{ for every }  Y \in \frak h,
$$
 or equivalently
\begin{equation} \label{othogonalh}
 \Omega (J [X, Y], X) = \Omega ([X, Y], JX),  \mbox{ for every } Y \in \frak h.
\end{equation}
Using $J (ad_X (Y)) = ad_X  (J Y)$ we have
$$
\Omega (J [X, Y], JX) = \Omega ([X, JY], JX).
$$
By \eqref{othogonalh} it follows that
$$
\Omega ( [X, JY], JX) = \Omega (J [X, JY], X) = - \Omega ([X, Y], X),
$$
i.e.  $g([X, Y], JX) =0$, for every $Y \in \frak h$.
By Section \ref{sol}, we have  the decomposition
\[\frak h_\C=V_{\bf 0}(\frak h_\C)\oplus V_{\alpha_{1}}(\frak h_\C)\oplus \dots\oplus V_{\alpha_{n}}(\frak h_\C)
\]
where $\bf 0$ is the trivial character and $\alpha_{1},\dots, \alpha_{n}$ are some non-trivial characters. Therefore
\[\g_\C=\langle X, JX\rangle \oplus  V_{\alpha_{1}}(\frak h_\C)\oplus \dots\oplus V_{\alpha_{n}}(\frak h_\C)
\]
with
\[[X, V_{\alpha_{i}}(\frak h_\C)]\subset V_{\alpha_{i}}(\frak h_\C), \quad [JX, V_{\alpha_{i}}(\frak h_\C)]=0
\]
and
\[JV_{\alpha_{i}}(\frak h_\C)\subset V_{\alpha_{i}}(\frak h_\C).
\]
 Thus
\[\bigwedge \g^{\ast}_\C= \Lambda \langle x,  Jx \rangle \otimes \Lambda ( V^{\ast}_{\bf 0}(\n_\C)\oplus V^{\ast}_{\alpha_{1}}(\g_\C)\oplus \dots\oplus V^{\ast}_{\alpha_{n}}(\g_\C) ),\]
where $x$ denotes the dual of $X$. Since $\g$ is not of type  (I), then there exists $\xi \in V_{\alpha_i}$ such that
$$
J \xi = i \xi, \quad d \xi = a_i \xi \wedge x + \beta_i \wedge x,
$$
with $Re (a_i) \neq 0$ and $\beta_i \in V_{\alpha_{i}}(\frak h_\C)$ such that $\beta_i \wedge \xi =0$. Therefore $x \wedge  \xi \wedge \overline \xi$  can appear only in $d (\xi \wedge \overline \xi)$, but this  implies then that $\Omega (Z, JZ) =0$, where $Z -  i JZ$ is the dual of $\xi$.
\end{proof}

\begin{rem}  {\em Theorem \ref{almost-Abilian} can be generalized to (I)-type Lie algebras by introducing some extra assumptions on $J$. Indeed, if $(\Omega,J)$ is a  Hermitian-symplectic structure   on a unimodular  almost-abelian Lie algebra $\g$ of type $I$, then we still have the  orthogonal decomposition with respect to the metric $g$ induced by $\Omega^{1,1}$
\begin{equation} \label{decomposalmab}
\g = {\mbox{span}}_{\R} \langle X, JX\rangle   \oplus  \, \frak h,
\end{equation}
with $[X, JX]�\in \frak h$, $\frak h$ abelian  and $ad_X (\frak h) \subseteq \frak h$. So in particular,
$\frak g^1 \subseteq \frak h$ and   $d x =0= d  (Jx)$. Therefore if for instance we require that
 $[X, JX] = 0$, then $\frak c =\langle X\rangle $ is an abelian complement of $\frak n$ and $J$ is ${\frak  c}$-invariant. So  if the associated simply-connected Lie group $G$ has a lattice, we can apply Theorem \ref{th4.2} obtaining that $(G/\Gamma, J)$ is K\"ahler.}
\end{rem}

Using Proposition \ref{almostabrealtype} and the previous remark we can prove the following

\begin{theorem} Let $G/\Gamma$ be a $6$-dimensional solvmanifold  endowed with   a left-invariant complex structure $J$.
If $G$ is almost abelian and $G/\Gamma$ admits a symplectic structure taming $J$, then $G/\Gamma$ admits a K\"ahler structure.
\end{theorem}

\begin{proof} If $G$ is not of type (I), then the result follows by  Proposition \ref{almostabrealtype}. Suppose  that $G$ is of type  (I). By  previous remark we have the orthogonal decomposition
\eqref{decomposalmab}  with $[X, JX]�\in \frak h$, $\frak h$ abelian  and $ad_X (\frak h) \subseteq \frak h$.

If  $[X, JX] =0$, the result follows applying  Theorem \ref{th4.2}. Suppose that $Y=[X, JX] \neq 0$. Since $Y \in \frak h$, we have that $X, JX, Y, JY$ are  linearly independent and they  generate a $4$-dimensional subspace of $\frak g$.

If $[X, Y] \in {\mbox{span}}_{\R} \langle Y, JY\rangle$, then $\frak k =  {\mbox{span}}_{\R} \langle X, JX, Y,JY\rangle $ is a  $4$-dimensional  Lie subalgebra of $\frak g$. Since $\frak k$ is $J$-invariant, then $\frak k$ admits a Hermitian-symplectic structure. The result follows from the fact the $\frak k$ is unimodular and then it  has to be isomorphic to $\tau \tau'_{3,0}$, but  if $[X, JX] \neq 0$ this  is not possible.

If $[X, Y] $ does not belong to ${\mbox{span}}_{\R}\langle Y, JY\rangle $, then $$\{ X, JX, Y=[X, JX], JY, Z = [X, Y], JZ \}$$  is a basis of $\frak g$. Note that $JZ =   [X, JY]$. Let $ \{ x, Jx, y, Jy, z, J z \}$ be the dual basis of $\{ X, JX, Y JY, Z , JZ \}$. We have that $\g$ has structure equations
$$
\left \{ \begin{array}{l}
d x =0,\\
d (Jx) =0,\\
d y = - x \wedge Jx,\\
d (Jy) = x \wedge (a z + b Jz),\\
d z = - x \wedge y,\\
d (Jz) = - x \wedge Jy,
\end{array}
\right.
$$
with $a, b \in \R$. Then, by a direct computation one has that
$$
d (z \wedge Jz) = - x \wedge y \wedge Jz + z \wedge x \wedge Jy
$$
and that the term $z \wedge x \wedge Jy$ can appear only in $d (z \wedge Jz)$. Therefore, we must have $\Omega (Z, JZ) =0$.
\end{proof}


\begin{thebibliography}{12}
\bibitem{ABD1} A. Andrada,  M. L. Barberis and  I.  Dotti, Classification of abelian complex structures on 6-dimensional Lie algebras, {\em J. Lond. Math. Soc.} (2) {\bf 83} (2011), no. 1, 232--255.
\bibitem{ABD} A. Andrada, M. L.  Barberis and I. Dotti, Abelian Hermitian Geometry, {\em Differential Geom. Appl.} {\bf 30} (2012),  509--519.
\bibitem{Aus}
L. Auslander, An exposition of the structure of solvmanifolds. I. Algebraic theory,   {\em Bull. Amer. Math. Soc.} {\bf 79}  (1973), no. 2, 227--261.

\bibitem{BDM} M. L. Barberis, I. G. Dotti and  R.J. Miatello,  On
certain locally homogeneous Clifford manifolds, \emph{Ann. Glob. Anal. Geom.} {\bf 13}
(1995), 513--518.

\bibitem{BR}
O. Baues and J. Riesterer, Virtually abelian K\"ahler and projective groups.
{\em Abh. Math. Semin. Univ. Hambg.} {\bf 81} (2011), no. 2, 191--213.

\bibitem{BDeG} D.  Burde and W.  A. de Graaf, Classification of Novikov algebras, {\em Appl. Algebra Engrg. Comm. Comput.} {\bf 24} (2013), no. 1, 1--15.

\bibitem{Cavalcanti}  G.  Cavalcanti, SKT geometry,  {\tt arXiv:1203.0493}.

\bibitem{Chu} B. Y. Chu,
Symplectic homogeneous spaces,
{\em Trans. Amer. Math. Soc.} {\bf 197} (1974), 145--159.

\bibitem{CFK}
S. Console, A. Fino and H. Kasuya, Modification and cohomology of solvmanifolds,  {\tt arXiv:1301.6042}.

\bibitem{Dek} K. Dekimpe,  Solvable Lie algebras, Lie groups and polynomial structures,  {\em Compositio Math.} {\bf 121} (2000), no. 2,  183--204.

\bibitem{dek} K.  Dekimpe,
Semi-simple splittings for solvable Lie groups and polynomial structures, {\em Forum Math.} {\bf 12} (2000), no. 1, 77--96.

\bibitem{Dix}
J. Dixmier, Cohomologie des algebres de Lie nilpotentes, {\em Acta Sci. Math. Szeged} {\bf 16 } (1955), 246--250.

\bibitem{DLZ}
T. Draghici, T.-J. Li and W. Zhang: On the $J$-anti-invariant cohomology of almost complex 4-manifolds,   {\em Q. J. Math.} {\bf 64} (2013), no. 1, 83--111.

\bibitem{EF} N. Enrietti and A. Fino, Special Hermitian metrics and Lie groups,  {\em Differential Geom. Appl.} {\bf 29}  (2011), suppl. 1, 211--219.

\bibitem{EFV}  N. Enrietti, A.  Fino and  L. Vezzoni, Tamed symplectic forms and strong K\"ahler with torsion metrics,  {\em J. Symplectic Geom. } {\bf 10}  (2012), no. 2, 203--223

\bibitem{FPS} A. Fino, M.  Parton and S.  Salamon, Families of strong KT structures in six dimensions, {\em Comment. Math. Helv.}  {\bf 79}  (2004), no. 2, 317--340.

\bibitem{FT}  A. Fino and A. Tomassini,  A survey on strong KT structures, {\em  Bull. Math. Soc. Sci. Math. Roumanie}  (N.S.) {\bf 52} (100) (2009), no. 2, 99--116.

\bibitem{Hasegawa}  K. Hasegawa, A note on compact solvmanifolds with K\" ahler structures,
{\em Osaka J. Math.}  {\bf 43}  (2006), no. 1, 131--135.


\bibitem{Kas} H. Kasuya, Minimal models, formality and hard Lefschetz properties of solvmanifolds with local systems, {\em  J. Differential Geom.}  {\bf 93} (2013), no 2, 269--298.


\bibitem{KVai} H. Kasuya, Vaisman metrics on solvmanifolds and Oeljeklaus-Toma manifolds,  {\em Bull. London Math. Soc.} {\bf 45} (2013),  15--26.

\bibitem{KH} H. Kasuya, Hodge symmetry and decomposition on non-K\"ahler solvmanifolds, {\em J. Geom. Phys.} {\bf 76} (2014), 61--65.

\bibitem{LZ}
T.-J. Li and W. Zhang, Comparing tamed and compatible symplectic cones and cohomological properties of almost complex manifolds,  {\em Comm. Anal. Geom.} {\bf 17} (2009), no. 4, 651--683.



\bibitem{OT}K. Oeljeklaus and M. Toma,  Non-K\"ahler compact complex manifolds associated to number fields,  {\em Ann. Inst. Fourier (Grenoble)} {\bf 55} (2005), no. 1, 161--171.

\bibitem{Ov} G. Ovando, Invariant pseudo-K\"ahler metrics in dimension four, {\em J.  Lie Theory} {\bf 16} (2006), 371--391.


\bibitem{OV} A. L. Onishchik and E. B. Vinberg (Eds), Lie groups and Lie algebras I
\hspace{-.1em}I\hspace{-.1em}I,  Encyclopaedia of Mathematical Sciences, 41. Springer-Verlag, Berlin, 1994.

\bibitem{P}
T. Peternell, Algebraicity criteria for compact complex manifolds, {\em Math. Ann.} {\bf 275}
(1986), no. 4, 653--672.

\bibitem{Sn} J. E. Snow, Invariant complex structures on four dimensional solvable real Lie groups, {\em Manuscripta Math.} {\bf 66} (1990),
397--412.

\bibitem{ST} J. Streets and  G. Tian, A Parabolic flow of pluriclosed metrics,  {\em Int. Math. Res. Not. IMRN} {\bf 2010} (2010), 3101--3133.

\bibitem{V} M. Verbitsky, Rational curves and special metrics on twistor spaces,  {\tt arXiv:1210.6725.}

\bibitem{Yam}
T. Yamada, Ricci flatness of certain compact pseudo-K\"ahler solvmanifolds, {\em J. Geom. Phys.} {\bf 62} (2012), no. 5, 1338--1345.

\bibitem{Z}
W. Zhang, From Taubes currents to almost K\"ahler forms, {\em Math. Ann.} {\bf 356} (2013), no. 9, 969--978.
\end{thebibliography}
\end{document}